\documentclass[11pt]{amsart}
\usepackage{palatino}
\usepackage{amsfonts}
\usepackage{amssymb}
\usepackage{amscd}
\usepackage[breaklinks,bookmarksopen,bookmarksnumbered]{hyperref}
\usepackage[dvips]{graphicx}

\newtheorem{theorem}{Theorem}[section]
\newtheorem{proposition}[theorem]{Proposition}
\newtheorem{corollary}[theorem]{Corollary}

\theoremstyle{definition}

\newtheorem{remark}[theorem]{Remark}

\newtheorem{conjecture/question}[theorem]{Conjecture/Question}

\newtheorem{remark/definition}[theorem]{Remark/Definition}
\newtheorem{terminology/notation}[theorem]{Terminology/Notation}

\def\PP{{\textbf P}}
\def\OO{\mathcal{O}}

\def\cD{\mathcal{D}}

\def\L{\mathcal{L}}

\def\cM{\mathcal{M}}

\def\cU{\mathcal{U}}

\def\Pic0{{\rm Pic}^0(X)}

\def\ff{\overline{\mathcal{F}}}
\def\mm{\overline{\mathcal{M}}}

\def\kk{\overline{\mathcal{K}}}
\def\dd{\overline{\mathcal{D}}}
\def\cc{\overline{\mathcal{C}}}

\newcommand{\dblq}{{/\!/}}

\begin{document}
\title{The classification of universal Jacobians  over the moduli space of curves}

\author[G. Farkas]{Gavril Farkas}

\address{Humboldt-Universit\"at zu Berlin, Institut F\"ur Mathematik,  Unter den Linden 6
\hfill \newline\texttt{}
 \indent 10099 Berlin, Germany} \email{{\tt farkas@math.hu-berlin.de}}
\thanks{}

\author[A. Verra]{Alessandro Verra}
\address{Universit\'a Roma Tre, Dipartimento di Matematica, Largo San Leonardo Murialdo \hfill
\indent 1-00146 Roma, Italy}
 \email{{\tt
verra@mat.unirom3.it}}

\maketitle

A well-established principle of Mumford  asserts that all  moduli spaces of curves of genus $g>2$ (with or without marked points), are varieties of general type, except a finite number of cases occurring for relatively small genus, when these varieties tend to be unirational, or at least uniruled, see \cite{HM}, \cite{EH1}, \cite{FL}, \cite{Log}, \cite{V} for illustrations of this fact. In all known cases, the transition from uniruledness to being of general type is quite sudden. The aim of this paper is to determine the Kodaira dimension of the universal Jacobian of degree $g$ over $\mm_g$ for any genus $g$, in particular to highlight the surprising transition cases $g=10, 11$.

Let $\cc_{g, n}:=\mm_{g, n}/\mathfrak S_n$ be the universal symmetric product of degree $n$. Since the fibre of the projection map $\varphi: \cc_{g, n}\rightarrow \mm_g$ over a smooth curve $[C]\in \cM_g$ is birational to the $n$-th symmetric product $C_n$, it  follows trivially that $\cc_{g, n}$ is uniruled when $n>g$. The global Abel-Jacobi map $$\mathfrak{a}_g: \cc_{g, g}\dashrightarrow \overline{\mathfrak{Pic}}_{g}^g,$$
establishes a birational isomorphism between $\cc_{g, g}$ and (a compactification of) the degree $g$ universal Picard variety $\alpha_g: \overline{\mathfrak{Pic}}_g^g\rightarrow \mm_g$. For a smooth curve $[C]\in \cM_g$, the map $\varphi^{-1}([C])\rightarrow \mathrm{Pic}^{g-1}(C)$
factors through $C_{g}$: the Abel-Jacobi map $C_g\rightarrow \mathrm{Pic}^g(C)$ is the blow-up of $\mathrm{Pic}^g(C)$ along the Fitting ideal corresponding to the subscheme $W^1_g(C)$, whereas $\varphi^{-1}([C])\rightarrow C_g$ is an iterated blow-up along the diagonals. Thus, we may regard $\cc_{g, g}$ as a global blow-up of $\overline{\mathfrak{Pic}}_g^g$. Applying the additivity of Kodaira dimensions for abelian fibrations \cite{U} to the fibre space $\alpha_g$, we obtain that $\kappa(\overline{\mathfrak{Pic}}_g^g)=3g-3$ whenever $\mm_g$ is of general type. It is natural to wonder whether the equality $\kappa(\overline{\mathfrak{Pic}}_g^g)=\kappa(\mm_g)$ holds for every $g$. We answer this question in the negative. Our most picturesque result, concerns the transition cases in the birational classification of universal Jacobians:
\begin{theorem}\label{picard}
The universal Jacobian $\overline{\mathfrak{Pic}}_{10}^{10}$ has Kodaira dimension zero. The Kodaira dimension of $\overline{\mathfrak{Pic}}_{11}^{11}$ equals $19$.
\end{theorem}
It is  well-known that both $\mm_{10}$ and $\mm_{11}$ are unirational varieties, see \cite{AC2} and \cite{CR} respectively. We note that when $g\leq 11$, these are the only cases when $\cc_{g, n}$ has non-negative Kodaira dimension. Using the existence of certain \emph{Mukai models} of the moduli space of curves of genus $g<10$,  it easily follows that $\overline{\mathfrak{Pic}}_{g}^g$ is unirational for $g\leq 9$. In fact more can be said:
\begin{theorem}\label{symprod}
The moduli space $\cc_{g, n}$ is unirational for all $g<10$ and $n\leq g$. Furthermore, $\cc_{10, n}$ is uniruled for all $n\neq 10$; the space $\cc_{11, n}$ is uniruled for all $n\neq 11$.
\end{theorem}
For higher $g$, we show that $\cc_{g, g}$ is of the maximal Kodaira dimension it could possibly have, in view of Iitaka's \emph{easy addition} inequality for fibre spaces
$$\kappa(\cc_{g, g})\leq \mbox{dim}(\mm_g)+\kappa\bigl(\varphi^{-1}([C])\bigr)=3g-3.$$
\begin{theorem}\label{jacgentype}
For $g>11$, the Kodaira dimension of $\overline{\mathfrak{Pic}}^{g}_{g}$ is equal to $3g-3$.
\end{theorem}
Theorems \ref{picard}, \ref{symprod}, \ref{jacgentype} highlight the fact that $\overline{\mathfrak{Pic}}_{g}^g$ does not capture the intricate transition of $\mm_g$ from uniruledness to general type that occurs in the range $17\leq g\leq 21$.
\vskip 4pt
We describe the main steps in the proof of Theorems \ref{picard}, \ref{symprod} and \ref{jacgentype}. A key role is played by the effective divisor
$$\cD_g:=\{[C, x_1, \ldots, x_g]\in \cM_{g, g}: h^0\bigl(C, \OO_C(x_1+\cdots+x_g)\bigr)\geq 2\}.$$
The  class of the closure of $\cD_g$ inside $\mm_{g, g}$ has been computed, see \cite{Log} Theorem 5.4, or \cite{F1} Theorem 4.6, for an alternative proof:
$$\dd_g\equiv -\lambda+\sum_{i=1}^g \psi_i-\sum_{i=0}^{[g/2]} \sum_{T\subset \{1, \ldots, g\}} {|\#(T)-i|+1\choose 2}\delta_{i: T}\in \mbox{Pic}(\mm_{g, g}).$$ The divisor $\dd_g$ is $\mathfrak S_g$-invariant under the action permuting the marked points, thus if $\pi:\mm_{g, g}\rightarrow \cc_{g, g}$ is the quotient map, there exists an effective divisor $\widetilde{\cD}_g\in \mbox{Eff}(\cc_{g, g})$ such that $\dd_g=\pi^*(\widetilde{\cD}_g)$. Note that $\widetilde{\cD}_g$ is an exceptional divisor of the rational Abel-Jacobi map $\mathfrak{a}_g:\cc_{g, g} \dashrightarrow \overline{\mathfrak{Pic}}_g^g$, and as such, it is uniruled and an extremal point of $\mbox{Eff}(\cc_{g, g})$ \footnote{We recall that if $f:X\dashrightarrow Y$ is a birational contraction between $\mathbb Q$-factorial normal projective varieties, the irreducible components of the exceptional locus of $f$ give rise to edges  of $\overline{NE}^1(X)$. In the case of the divisor $\widetilde{\cD}_g$ we shall also prove this assertion directly using the geometry of $\cc_{g, g}$, see Proposition \ref{ajpencil}. }.

In Theorem \ref{sing}, we prove that for $g\geq 4$ pluri-canonical forms on $\cc_{g, g, \mathrm{reg}}$ extend to any desingularization of $\cc_{g, g}$. Thus in order to bound $\kappa(\cc_{g, g})$ from below, it suffices to exhibit sufficiently many global sections of $K_{\cc_{g, g}}$.  To that end, we choose an effective divisor class $D\equiv a\lambda-\sum_{i=0}^{[g/2]}b_i\delta_i \in \mbox{Eff}(\mm_g)$ of small slope $s(D):=a/\mbox{min}_{i=0}^{[g/2]} b_i$. For composite $g+1$, one can take $D=\mm_{g, d}^r$ to be the closure of the Brill-Noether
divisor of curves with a $\mathfrak g^r_d$, where $\rho(g, r, d)=g-(r+1)(g-d+r)=-1$; there exists a constant $c_{g, r, d}>0$, such that \cite{EH2},
$$\mm_{g, d}^r\equiv c_{g, r, d}\cdot  \mathfrak{bn}_g:\equiv c_{g, r, d}\Bigl((g+3)\lambda-\frac{g+1}{6}\delta_0-\sum_{i=1}^{[g/2]} i(g-i)\delta_i\Bigr)\in \mathrm{Pic}(\mm_g),$$
where the previous formula is used to define the class $\mathfrak{bn}_g$ proportional to \emph{all} Brill-Noether divisors on $\mm_g$.
In particular $s(\mm_{g, d}^r)=6+12/(g+1)$. By linear interpolation, we find  an effective divisor $E$ on $\cc_{g, g}$ supported  along $\widetilde{\cD}_g$, $\varphi^*(D)$ and the boundary of $\cc_{g, g}$, such that $$K_{\cc_{g, g}}=\bigl(14-2s(D)\bigr)\ \varphi^*(\lambda)+E.$$ Whenever $s(D)<7$ (and such a divisor $D\subset \mm_g$ can be chosen exactly when $g\geq 12$, see \cite{FV} Theorem 6.1 for the particularly difficult case $g=12$), the following inequality holds:
$$\kappa(\cc_g)\geq \kappa\Bigl(\cc_g, (14-2s(D))\varphi^*(\lambda)\Bigr)=\kappa(\mm_g, \lambda)=3g-3.$$
Since the opposite inequality is immediate, this proves Theorem \ref{jacgentype}. We summarize this discussion by linking $\cc_{g, g}$ to the slope $s(\mm_g):=\mbox{inf}\{s(D): D\in \mathrm{Eff}(\mm_g)\}$ of the moduli space of curves.
\begin{proposition} \label{slope}
Assume $s(\mm_g)<7$ for a given genus $g$. Then $\kappa(\cc_{g, g})=3g-3$.
\end{proposition}
This highlights that the birational geometry of $\cc_{g, g}$ is governed by a linear series on $\mm_g$ of slope $7$, rather than the canonical linear series, for which $s(K_{\mm_g})=13/2$.
\vskip 4pt

We discuss some of the cases $g\leq 11$. When $7\leq g\leq 9$, there exists a birational contraction of $\dd_g$ under a birational map $$f_g: \mm_{g, g} \dashrightarrow \mathfrak{M}_{g, g},$$ where  $\mathfrak{M}_{g, g}$ is a GIT model of $\mm_{g, g}$ emerging from Mukai's constructions of canonical curves as sections of homogeneous varieties \cite{M1}, \cite{M3}, \cite{M4}. In genus $11$, we prove a stronger result, concerning the birational type of both $\cc_{11, 11}$ and its covering $\mm_{11, 11}$:


\begin{theorem}\label{gen11}
One has that $\kappa(\mm_{11, 11})=\kappa(\cc_{11, 11})=19$.
\end{theorem}

Note that $\mbox{dim}(\mm_{11, 11})=41$, and $\mm_{11, 11}$ is the first example of a moduli space $\mm_{g, n}$ with $g\geq 2$, having intermediate Kodaira dimension. It is known \cite{Log} that $\mm_{11, n}$ is uniruled for $n\leq 10$ and of general type for $n\geq 12$.  To interpret such results, for a genus $g\geq 2$ we define the invariant
$$\zeta(g):=\mathrm{min}\{n\in \mathbb Z_{\geq 0}: \kappa(\mm_{g, n})\geq 0\}.$$
We think of $\zeta(g)$ as measuring the \emph{complexity} of the general curve of genus $g$. It is known that the relative dualizing sheaf of the forgetful map $\mm_{g, n}\rightarrow \mm_{g, n-1}$ is big, see for instance \cite{CHM} p. 19, thus it follows that $\mm_{g, n}$ is of general type for $n>\zeta(g)$. Then results  in \cite{HM}, \cite{EH2}, \cite{F2} imply that $\zeta(g)=0$, for $g\geq 22$. From \cite{FP} Proposition 7.5 one obtains the value $\zeta(10)=10$, whereas  Theorem \ref{gen11} implies that $\zeta(11)=11$.  This indicates, in precise terms, that  counter-intuitively, algebraic curves of genus $10$ are more complicated than curves of genus $11$!
\vskip 3pt

We make a few comments on Theorem \ref{gen11}. We note that  $K_{\mm_{11, 11}}$ is an effective combination of the pull-back to $\mm_{11, 11}$ of the $6$-gonal  divisor $\mm_{11, 6}^1$  on $\mm_{11}$, the divisor $\dd_{11}$, and certain boundary classes $\delta_{i:S}$.  Then we construct (cf. Proposition \ref{extrem}), rational curves $R\subset \mm_{11, 11}$ passing through a general point of $\dd_{11}$, such that (i) $-R\cdot \dd_{11}>0$ equals precisely the multiplicity of $\dd_{11}$ in the above mentioned expression of $K_{\mm_{11, 11}}$, and (ii) $R$ is disjoint from all boundary divisors $\Delta_{i: T}$.  Therefore, $n\dd_{11}$ is a fixed component of the pluri-canonical linear series $|nK_{\mm_{11, 11}}|$ for all $n\geq 1$. The equality  $\kappa(\mm_{11, 11})=19$ is related to the Mukai  fibration
$$q_{11}: \mm_{11, 11}\dashrightarrow \ff_{11},$$ over the $19$-dimensional moduli space $\ff_{11}$ of polarized $K3$ surfaces of degree $20$. The map $q_{11}$ associates to a general element $[C, x_1, \ldots, x_{11}]\in \mm_{11, 11}$ the unique $K3$ surface $S$ containing $C$, see \cite{M2}. According to Mukai,  $S$ is precisely the Fourier-Mukai dual $K3$ surface to the non-abelian Brill-Noether locus corresponding to semistable vector bundles of rank $2$
$$S^{\vee}=\mathcal{SU}_C(2, K_C, 7):=\{E\in \mathcal{SU}_C(2, K_C): h^0(C, E)\geq 7\}.$$

An analysis of the fibration $q_{11}$ shows that the difference $K_{\mm_{11, 11}}-\dd_{11}$ is essentially the pull-back of an ample class on $\ff_{11}$. Eventually, this leads to the equality $\kappa(\mm_{11, 11})=\kappa(\mm_{11}, \mathfrak{bn}_{11})=19$, where the last symbol stands for the Iitaka dimension of the linear system generated by the Brill-Noether divisors $\mm_{11, 6}^1$ and $\mm_{11, 9}^2$ on $\mm_{11}$.
\vskip 4pt

In the final section of this paper we study the uniruledness of $\mm_{g, n}$ when $g\leq 9$.
\begin{theorem}\label{genul8}
The space $\mm_{g, n}$ is uniruled for $n\leq f(g)$, where $f(g)$ is given in the table below:
\begin{center}
\begin{tabular}{c|cccccc}
$g$ & 5 & 7 & 8 & 9 & 10&
\\
\hline
$f(g)$ & 13 & 13 & 12 & 10& 9&
\\
\hline $h(g)$ & 15 & 15 & 14 & 13& 11&

\end{tabular}
\end{center}
\end{theorem}

New here is the statement regarding the uniruledness of $\mm_{g, n}$. For the sake of comparison, we have copied from \cite{Log} and \cite{F1} Theorem 1.10, the bound $h(g)$, for which $\mm_{g, n}$ is known to be of general type when $n\geq h(g)$. We note that moreover, $\kappa(\mm_{7, 14})\geq 2$, thus $\zeta(7)=14$. Nothing appears to be known about the Kodaira dimension of $\mm_{5, 14}$ and $\mm_{8, 13}$, which are the missing cases from the classification, when $g\leq 8$. The case $g=6$, where a complete solution is known, cf.  \cite{Log}, \cite{CF}, is omitted from the table.

In order to prove Theorem \ref{genul8}, it suffices to establish that $K_{\mm_{g, n}}$ is not pseudo-effective and then use \cite{BDPP}  (see also \cite{L} Section 11.4.C for a transparent presentation), to conclude that $\mm_{g, n}$ is uniruled. The non-effectiveness of $K_{\mm_{g, n}}$ is established by exhibiting \emph{one or two extremal uniruled} divisors on $\mm_{g, n}$, satisfying certain numerical properties, see Proposition \ref{uniruled}.

\vskip 3pt
We close the introduction by pointing out that, shortly after this paper was published on arXiv, the paper \cite{BFV} by Bini, Fontanari and Viviani appeared, where  the Kodaira dimension of  Caporaso's \cite{C} compactification $P_{d, g}$ of the universal Picard variety $\mathfrak{Pic}_g^d$ is determined for all $g\geq 4$ and degrees $d$ such that $\mbox{gcd}(d+g-1, 2g-2)=1$. In the case $d=g$, the result from \cite{BFV} overlaps with our Theorem 0.3. Their methods are however quite different from ours. The difference lies in that \emph{two different} birational models are used to compactify the universal Jacobian. These models are related via the birational Abel-Jacobi map $\mathfrak{a}_g:\cc_{g, g}\dashrightarrow P_{g, g}$, which is a blow-up of a codimension $2$ locus in $P_{g, g}$ and has $\widetilde{\cD}_g$ as its exceptional divisor. This also explains the discrepancy between the canonical classes $K_{\cc_{g, g}}$ and $K_{P_{g, g}}$ (cf. \cite{BFV} Theorem 1.4) respectively. In view of \cite{BFV}, we would also like to point out the existence of our earlier preprint \cite{FV2}. In the second half of \cite{FV2} (which will not be published  having been incorporated in  Sections 3-5 of this paper), the birational classification of $\mathfrak{Pic}_d^g$ for $g\leq 11$ is carried out to a very large extent. In particular, the intermediate case $g=11$ and the connection to the Mukai fibration $q_{11}:\mm_{11, 11}\dashrightarrow \overline{\mathcal{F}}_{11}$  appears already in \cite{FV2} Theorem 0.3.

\section{An extremal effective divisor on $\mm_{g, g}$}

We begin by setting notation and terminology. If $\bf{M}$ is a Deligne-Mumford stack we denote by $\cM$ its associated coarse moduli space. Whenever we refer to the Picard group of one of the moduli spaces $\cM\in \{\mm_{g, n}, \cc_{g, n}\}$, to keep things simple we denote $\mbox{Pic}(\cM):=\mbox{Pic}(\cM)_{\mathbb Q}=\mbox{Pic}(\bf{M})_{\mathbb Q}$ for the rational Picard group. This allows us to make no distinction between divisor classes on the stack and the associated coarse moduli space.

Let $X$ be a complex $\mathbb Q$-factorial variety. A $\mathbb Q$-Weil divisor $D$ on $X$ is said to be \emph{movable} if $\mbox{codim}\bigl(\bigcap_{m} \mbox{Bs}|mD|, X\bigr)\geq 2$, where the intersection is taken over all $m$ which are sufficiently large and divisible. We say that $D$ is \emph{rigid} if $|mD|=\{mD\}$, for all $m\geq 1$ such that $mD$ is an integral Cartier divisor. The \emph{Kodaira-Iitaka dimension} of a divisor $D$ on $X$ is denoted by $\kappa(X, D)$. As usual, we set $\kappa(X):=\kappa(X, K_X)$. We say that a curve $\Gamma\subset X$ is a \emph{covering curve} for a divisor $D\subset X$, when $\Gamma$ deforms in a family of $1$-cycles $\{\Gamma_t\}_{t\in T}$, such that $\overline{\cup_{t\in T} \Gamma_t}=D$.


We recall the notation for boundary divisor classes on the moduli space $\mm_{g, n}$, cf. \cite{AC1}.  For an integer $0\leq i\leq [g/2]$ and a set of labels
$T\subset \{1, \ldots, n\}$, we denote by $\Delta_{i: T}$ the closure in $\mm_{g, n}$ of the locus of $n$-pointed curves $[C_1\cup C_2, x_1, \ldots, x_n]$, where $C_1$ and $C_2$ are smooth curves of genera $i$ and $g-i$ respectively, and the marked points lying on $C_1$ are precisely those labeled
by $T$. As usual, we define $\delta_{i: T}:=[\Delta_{i: T}]\in \mbox{Pic}(\mm_{g, n})$. For $0\leq i\leq [g/2]$ and $0\leq c\leq g$, we set
$$\Delta_{i: c}:=\sum_{\#(T)=c}\delta_{i: T}, \ \ \ \  \delta_{i: c}:=[\Delta_{i: c}]_{\mathbb Q}\in \mbox{Pic}(\mm_{g, n}).$$
 By convention, $\delta_{0: c}:=\emptyset$, for $c<2$.
If $\phi:\mm_{g, n}\rightarrow \mm_g$ is the morphism forgetting the marked points, we set $\lambda:=\phi^*(\lambda)\in \mbox{Pic}(\mm_{g, n})$ and $\delta_{\mathrm{irr}}:=\phi^*(\delta_{\mathrm{irr}})\in \mbox{Pic}(\mm_{g, n})$, where $\delta_{\mathrm{irr}}:=[\Delta_{\mathrm{irr}}]\in \mbox{Pic}(\mm_g)$ denotes the class of the locus of irreducible nodal curves. Furthermore, $\psi_1, \ldots, \psi_n\in \mbox{Pic}(\mm_{g, n})$ are the cotangent classes corresponding to the marked points. The canonical class of $\mm_{g, n}$ can be computed by using Grothendieck-Riemann-Roch for the universal curve over $\mm_{g, n}$:
\begin{equation}\label{canmgn}
K_{\mm_{g, n}}\equiv 13\lambda-2\delta_{\mathrm{irr}}+\sum_{i=1}^n \psi_i-2\sum_{T\subset \{1, \ldots, n\}\atop i\geq 0} \delta_{i: T}-\delta_{1: \emptyset}.
\end{equation}
On the universal symmetric product $\cc_{g, n}$, we denote by $\widetilde{\lambda}, \widetilde{\delta}_{\mathrm{irr}}, \widetilde{\delta}_{i: c}:=[\widetilde{\Delta}_{i: c}]\in \mbox{Pic}(\cc_{g, n})$ the divisor classes corresponding to the same symbols on $\mm_{g, n}$. The general point from the $\widetilde{\Delta}_{0: 2}$ corresponds to a marked curve with automorphism  group isomorphic to  $\mathbb Z/2\mathbb Z$, therefore $\widetilde{\delta}_{0: 2}=[\widetilde{\Delta}_{0: 2}]_{\mathbb Q}=[\widetilde{\Delta}_{0, 2}]/2$. Let $\pi:\mm_{g, n}\rightarrow \cc_{g, n}$ and $\varphi:\cc_{g, n}\rightarrow \mm_g$ be the quotient and forgetful maps respectively, thus $\phi=\varphi \circ \pi$. Clearly, $\pi^*(\widetilde{\lambda})=\lambda, \ \pi^*(\widetilde{\delta}_{\mathrm{irr}})=\delta_{\mathrm{irr}}$, $\pi^*(\widetilde{\delta}_{i: c})=\delta_{i: c}$, where the last formula, in the case $i=0, c=2$, takes into account the branching of the map $\pi$ along the divisor $\widetilde{\Delta}_{0: 2}\subset \cc_{g, n}$.

We introduce the tautological line bundle $\mathbb L$ on $\cc_{g, n}$, having fibre $$\mathbb L[C, x_1+ \cdots+ x_n]=T_{x_1}^{\vee}(C)\otimes \cdots \otimes T_{x_n}^{\vee}(C),$$ over a point $[C, x_1+\cdots+x_n]:=\pi([C, x_1, \ldots, x_n])\in \cc_{g, n}$. We set $\widetilde{\psi}:=c_1(\mathbb L)$, and let $\pi_i:\mm_{g, n}\rightarrow \mm_{g, 1}$ be the morphism forgetting all expect the $i$-th marked point. Then
\begin{equation}\label{tildepsi}
\pi^*(\widetilde{\psi})=\sum_{i=1}^n \pi_i^*(\psi)=\sum_{i=1}^n \Bigl(\psi_i-\sum_{i\in T\subset \{1, \ldots, n\}} \delta_{0: T}\Bigr)=\sum_{i=1}^n \psi_i-\sum_{c=2}^n c\ \delta_{0:c}\in \mathrm{Pic}(\mm_{g, n}).
\end{equation}
From the Riemann-Hurwitz formula $K_{\mm_{g, n}}=\pi^*(K_{\cc_{g, n}})+\delta_{0: 2}$ applied to the quotient map $\pi$, after observing that the pull-back map
$\pi^*: \mbox{Pic}(\cc_{g, n})\rightarrow \mbox{Pic}(\mm_{g, n})$ is injective, we obtain the formula:
\begin{equation}\label{canonicsym}
K_{\cc_{g, n}}\equiv 13\widetilde{\lambda}-2\widetilde{\delta}_{\mathrm{irr}}+\widetilde{\psi}-2\sum_{i\geq 1, c\geq 0\atop (i, c)\neq (1, 0)} \widetilde{\delta}_{i: c}-3\widetilde{\delta}_{1: 0}-\widetilde{\delta}_{0:2}+\sum_{c=3}^n (c-2)\ \widetilde{\delta}_{0: c}\in \mathrm{Pic}(\cc_{g, n}).
\end{equation}
\vskip 3pt

In order to obtain lower bounds on the Kodaira dimension of $\cc_{g, n}$, we need to control its singularities. We fix a
point $[C, x_1+\cdots+x_n]\in  \cc_{g, n}$, and denote by $\mathrm{Def}(C, x_1, \ldots, x_n)$ the versal deformation space of the $n$-pointed curve $(C, x_1, \ldots, x_n)$, viewed as an open neighborhood of the origin in the tangent space to the moduli stack $T_{[C, x_1, \ldots, x_n]}(\overline{\textbf{M}}_{g, n})=\mbox{Ext}^1(\Omega_C^1, \OO_C(-x_1-\cdots-x_n))$. We set
$$\mathrm{Aut}(C, \bar{x}):=\bigl\{\sigma\in \mathrm{Aut}(C): \sigma(\{x_1, \ldots, x_n\})=\{x_1, \ldots, x_n\}\bigr\}.$$  An analytic neighbourhood of $[C, x_1+\cdots+x_n]\in \cc_{g, n}$ is isomorphic to the space
$$\mathrm{Def}(C, x_1, \ldots, x_n)/\mathrm{Aut}(C, \bar{x})\subset H^0\bigl(C, \omega_C\otimes \Omega_C^1(x_1+\cdots+x_n)\bigr)^{\vee}/\mathrm{Aut}(C, \bar{x}),$$
where the last identification uses Serre duality. To describe the action of $\mathrm{Aut}(C, \bar{x})$ on the tangent space of the moduli stack, we recall the concept of age.

Let $(V, \rho)$ be a  finite dimensional complex representation of a finite group $G$. If the eigenvalues of $\rho(g)\in GL(V)$
are $\mbox{exp}(2\pi i r_i)$,  where $0\leq r_i<1$ for $i=1, \ldots, d$, then following \cite{R}, we define \emph{the age} of the element $g\in G$ as $$\mathrm{age}(g):=r_1+\cdots+r_d.$$ According to the \emph{Reid-Tai criterion}, the singularities of $V/G$ are canonical, if and only if for each element $g\in G$ which does not act as a quasi-reflection, the inequality $\mathrm{age}(g)\geq 1$ holds, cf. \cite{HM} pg. 27. Next we show that the singularities of $\cc_{g, n}$ are no worse that those of $\mm_g$. In particular, one can bound the Kodaira dimension of $\cc_{g, n}$ by bounding the number of global sections of $K_{\cc_{g, n}}$.

\begin{theorem}\label{sing}
Fix integers $g\geq 4$ and $n, l\geq 0$, and let $\epsilon:\widetilde{\mathcal{C}}_{g, n}\rightarrow \cc_{g, n}$ be any resolution of singularities. Then $l$-canonical forms of $\cc_{g, n, \mathrm{reg}}$ extend, that is, there are group isomorphisms
$$\epsilon^*: H^0\bigl(\cc_{g, n, \mathrm{reg}}, K_{\cc_{g, n}}^{\otimes l}\bigr)\stackrel{\cong}\rightarrow H^0\bigl(\widetilde{\mathcal{C}}_{g, n}, K_{\widetilde{\mathcal{C}}_{g, n}}^{\otimes l}\bigr).$$
\end{theorem}
\begin{proof} We choose a point $[C, x_1+\cdots +x_n]\in \cc_{g, n}$ which violates the Reid-Tai criterion, that is, there exists an automorphism
$\sigma\in \mathrm{Aut}(C, \bar{x})$ which permutes the points $x_1, \ldots, x_n$, such that with respect to the action of $\sigma$ on $H^0\bigl(C, \omega_C\otimes \Omega_C^1(x_1+\cdots+x_n)\bigr)$, we have that \ $\mathrm{age}(\sigma)<1$. If $\{C_{\alpha}\}_{\alpha}$ are the normalizations of the components of $C$ and $\{p_{\alpha \beta}\}_{\beta}$ are the points on $C_{\alpha}$ whose images in $C$ are either nodes of $C$ or marked points $x_1, \ldots, x_n$, we recall that there exists an exact sequence:

\begin{equation}\label{torsion}
0\rightarrow \bigoplus_{p\in \mathrm{Sing}(C)} \mathrm{Tor}_p\rightarrow H^0\bigl(C, \omega_C\otimes \Omega_C^1(x_1+\ldots+x_n)\bigr)\rightarrow \bigoplus_{\alpha} H^0\bigl(C_{\alpha}, \omega_{C_{\alpha}}^{\otimes 2}(\sum_{\beta} p_{\alpha \beta})\bigr)\rightarrow 0,
\end{equation}
where $\mathrm{Tor}_p\subset H^0\bigl(C, \omega_C\otimes \Omega_C^1(x_1+\cdots+x_n)\bigr)$ is the $1$-dimensional space of torsion differentials based at $p\in \mathrm{Sing}(C)$. It is proved in \cite{HM} pg. 34, that $\mathrm{Tor}_p$ contributes at least $1/\mbox{ord}(\sigma)$ to $\mathrm{age}(\sigma)$, for each node $p\in \mathrm{Sing}(C)$. We distinguish two cases:

\noindent
(i) $\sigma$ acts non-trivially only on \emph{exceptional components} $R$ of $C$, which are smooth rational curves such that $\#(R\cap \overline{(C-R)})\leq 2$. In other words, $\sigma$ induces the trivial automorphism on the stable model of $C$. Let $R$ be an exceptional component, and for simplicity we assume that $R$ meets the rest of $C$ at only one point. We set $\{p\}:=R\cap C$ and denote by $P\subset R-\{p\}$ the marked points lying on $R$. Since $\sigma\in \mathrm{Aut}(R)$ has finite order, say $l$, one finds that $\sigma$ has precisely two fixed points $p=0, \infty\in R$, and that $\sigma(z)=\zeta\cdot z$, where $\zeta\neq 1$ is an $l$-th root of unity. The points in $P-\{\infty\}$ can be grouped in orbits of $l$ elements, and an immediate calculation shows that the contribution to $\mathrm{age}(\sigma)$ coming from $H^0(R, \omega_R^{\otimes 2}(p+P))$ is at least $(l-1)/l+{l-3\choose 2}$. Since as mentioned above, there is a further contribution to $\mbox{age}(\sigma)$ of at least $1/l$, coming from $\mathrm{Tor}_p(\omega_C\otimes \Omega_C^1 (x_1+\cdots+x_n))$, it follows that $\mathrm{age}(\sigma)\geq 1$, and this case corresponds to a canonical singularity. The case $\#(R\cap \overline{(C-R)})=2$ is analogous, also leading to a canonical singularity.
\vskip 3pt
\noindent
(ii) There exists a \emph{non-exceptional component} of $C$ on which $\sigma$ acts non-trivially. Since $\mathrm{age}(\sigma_C)$ (with respect to the action on $H^0(C, \omega_C\otimes \Omega_C^1)$) cannot exceed $\mbox{age}(\sigma)<1$, the analysis from \cite{HM} pg. 34-40, shows that in this case $C=C_1\cup E$,\ $C_1\cap E=\{p\}$, where $C_1$ is a smooth curve of genus $g-1$ and $E$ is an elliptic curve. Moreover $\sigma_{C_1}=\mbox{Id}_{C_1}$, and one distinguishes between the cases when $\mbox{ord}(\sigma_E)=2, 4, 6$. If at least one of the points $x_i$ lies on $E$, an immediate calculation shows that $\mathrm{age}(\sigma)\geq 1$, thus this case too corresponds to a canonical singularity. When $\{x_1, \ldots, x_n\}\subset C_1-\{p\}$, then if $U\subset \mm_g$ is the analytic neighbourhood of $[C_1\cup_p E]$ constructed in \cite{HM} pg. 41-43,  any pluri-canonical form defined on $\cc_{g, n, \mathrm{reg}}$ extends over $\varphi^{-1}(U)$. This completes the proof.
\end{proof}

We turn to the study of the divisor $\widetilde{\cD}_g\subset \cc_{g, g}$ defined as the closure of the locus of points $[C, x_1+\cdots+x_g]\in \mathcal{C}_{g, g}$ moving in a pencil. First we note that its class is given by
\begin{equation}\label{dgtilde}
\widetilde{\cD}_g\equiv -\widetilde{\lambda}+\widetilde{\psi}-\sum_{i\geq 1, c\geq 0} {|c-i|+1\choose 2}\widetilde{\delta}_{i: c}-\sum_{c=2}^g {c\choose 2}\widetilde{\delta}_{0: c} \in \mathrm{Pic}(\cc_{g, g}).
\end{equation}

We construct rational curves $\ell\subset \cc_{g, g}$ sweeping-out the divisor $\widetilde{\cD}_g$: We fix a curve $[C]\in \cM_g$, a complete base point free pencil $A\in W^1_g(C)$, and define $\ell\subset \cc_{g, g}$ as being the closure in moduli of the locus
$$\{[C, x_1+\cdots +x_g]\in \mathcal{C}_{g, g}: h^0(C, A(-x_1-\cdots -x_g))\geq 1\}\subset \mathcal{C}_{g, g}.$$
\begin{proposition}\label{ajpencil}
 One has that $\ell\cdot \widetilde{\psi}=2g-2$, \ $\ell\cdot \widetilde{\delta}_{0: 2}=2g-1$, whereas $\ell$ has intersection number $0$ with all remaining standard generators of $\mathrm{Pic}(\cc_{g, g})$. It follows that $\ell\cdot K_{\cc_{g, g}}=\ell\cdot \widetilde{\cD}_g=-1$. It follows that $\widetilde{\cD}_g$ is an extremal non-movable divisor in $\mathrm{Eff}(\cc_{g, g})$.
\end{proposition}
\begin{proof} Let $\widetilde{\ell}\subset C_g$, be the isomorphic image of $\ell$ under the blow-up map $\varphi^{-1}([C])\rightarrow C_g$ of the diagonals. Then, using e.g. \cite{K} Proposition 2.6, we have that $\widetilde{\ell}\cdot K_{C_g}=-1$. On the other side, $\widetilde{\ell}\cdot K_{C_g}=\ell \cdot K_{\cc_{g, g}}=\ell\cdot \widetilde{\psi}-\ell\cdot \widetilde{\delta}_{0: 2}$ (one may assume that  $A\in W^1_g(C)$ has only simple ramification points, hence $\ell\cdot \widetilde{\delta}_{0: c}=0$ for $c\geq 3$). Furthermore, $\ell\cdot \widetilde{\delta}_{0:2}$ equals the half of the number of ramification points of $A$, that is, $2g-1$, and the rest is immediate.
\end{proof}

As explained in the Introduction, for $g\geq 12$ the estimate $s(\mm_g)<7$ holds, and from (\ref{dgtilde}) one finds that
\begin{equation}\label{canontil}
K_{\cc_{g, g}}\in \mathbb Q_{>0}\Bigl\langle \widetilde{\lambda},\  [\widetilde{\cD}_g], \ \{\widetilde{\delta}_{i: c}\}_{(i, c)\neq (0, 2)}, \ \varphi^*\mathrm{Eff}(\mm_g)\Bigr\rangle.
\end{equation}
Coupled with Theorem \ref{sing}, this implies that $\kappa(\cc_{g, g})=3g-3$. Furthermore, we note that $\widetilde{\cD}_g$ appears with multiplicity $1$ in the stable base locus of $K_{\cc_g}$.
\begin{proposition}
Set $g\geq 11$. Then $|nK_{\cc_{g, g}}|=n\widetilde{\cD}_{g}+|nK_{\cc_{g, g}}-n\widetilde{\cD}_g|$, for all $n\geq 1$.
\end{proposition}
\begin{proof} The coefficient of $\widetilde{\cD}_g$ in the expression (\ref{canontil}) is equal to $1$. Since $\ell\subset \widetilde{\cD}_g$ is a covering curve, such that $\ell\cdot K_{\cc_{g, g}}=\ell\cdot \widetilde{\cD}_g=-1$, whereas $\ell$ has intersection number zero with the remaining classes appearing in (\ref{canontil}), the conclusion follows.
\end{proof}

We are thus left with the study  of  $\dd_g:=\pi^*(\widetilde{\cD}_g)\subset \mm_{g, g}$,  in the range $g\leq 11$.

\begin{proposition}\label{extrem}
For $3\leq g\leq 11$, the irreducible divisor $\dd_g$ is filled up by  rational curves $R\subset \mm_{g, g}$ such that $R\cdot \dd_g<0$. It follows that $\dd_g\in \mathrm{Eff}(\mm_{g, g})$ is an extremal rigid divisor. Moreover, when $g\neq 10$, one can assume that $R\cdot \delta_{i: T}=0$ for all $i\geq 0$ and $T\subset \{1, \ldots, g\}$.
\end{proposition}
\begin{proof}
We first treat the case $g\neq 10$, and start with a general point $[C, x_1, \ldots, x_g]\in \cD_g$. We assume that the points $x_1, \ldots, x_g\in C$ are distinct and $h^0(C, K_C(-x_1-\cdots -x_g))=1$. Let us consider the $(g-2)$-dimensional linear space $$\Lambda:=\langle x_1, \ldots, x_g\rangle\subset \PP\bigl(H^0(C, K_C)^{\vee}\bigr)=\PP^{g-1}.$$ Since $\phi(\cD_g)=\cM_g$, we may assume that $[C]\in \cM_g$ is a general curve. In particular, $C$ lies on a $K3$ surface $S\stackrel{|\OO_S(C)|}\hookrightarrow \PP^g$, which admits the canonical curve $C$ as a hyperplane section, cf. \cite{M1}. We intersect $S$ with the pencil of hyperplanes $\{H_{\lambda}\in (\PP^g)^{\vee}\}_{\lambda\in \PP^1}$ such that $\Lambda\subset H_{\lambda}$. Since (i) the locus of hyperplanes  $H\in (\PP^g)^{\vee}$ such that the intersection $S\cap H$ is not nodal has codimension $2$ in $(\PP^g)^{\vee}$, \ and (ii) the pencil $\{H_{\lambda}\}_{\lambda\in \PP^1}$ can be viewed as a general pencil of hyperplanes containing $\PP\bigl(H^0(C, K_C)^{\vee}\bigr)$ as a member, we may assume that all the curves $H_{\lambda}\cap S$ are nodal and that the nodes stay away from the fixed points $x_1, \ldots, x_g$. In this way we obtain a family in $\mm_{g, g}$
$$R:=\{[C_{\lambda}:=H_{\lambda}\cap S,\  x_1, \ldots, x_g]:\Lambda\subset H_{\lambda}, \ \lambda \in \PP^1\},$$
inducing a fibration $f:\tilde{S}:=\mbox{Bl}_{2g-2}(S)\rightarrow \PP^1$, obtained by blowing-up the base points of the pencil, together with $g$ sections given by the exceptional divisors
$E_{x_i}\subset \tilde{S}$ corresponding to the base points $x_1, \ldots, x_g$. The numerical parameters of $R$ are computed using, for instance, \cite{FP} Section 2. Precisely, one writes that
\begin{equation}\label{numericalparameters}
R\cdot \lambda=(\phi_*(R)\cdot \lambda)_{\mm_g}=g+1,\ \  R\cdot \delta_{\mathrm{irr}}=(\phi_*(R)\cdot \delta_{\mathrm{irr}})_{\mm_g}=6g+18, \ \ R\cdot \delta_{i: T}=0,
\end{equation}
for $i\geq 0$ and $T\subset \{1, \ldots, g\}$. Finally, from the adjunction formula, $R\cdot \psi_i=-(E_{x_i}^2)_{\tilde{S}}=1$ for $1\leq i\leq g$. Thus, $R\cdot \dd_g=-1$. Since $R$ is a covering
curve for the divisor $\dd_g$, it follows that $\dd_g$ is a rigid divisor on $\mm_{g, g}$.
\vskip 5pt

We turn to the case $g=10$, when the previous argument breaks down because the general curve $[C]\in \cM_{10}$ no longer lies on a $K3$ surface, see \cite{M1} Theorem 0.7.
More generally, we fix a genus $g<11, g\neq 9$ and pick a general point $[C, x_1, \ldots, x_{g}]\in \cD_{g}$. We denote by $X:=C_{ij}$ the nodal curve obtained from $C$ by identifying $x_i$ and $x_j$, where $1\leq i<j\leq g$. Since $[X]\in \Delta_0\subset \mm_{g+1}$ is a general $1$-nodal curve of genus $g+1$, using e.g. \cite{FKPS}, there exists a smooth $K3$ surface $S$ containing $X$.  We denote by $\nu:C\rightarrow X\subset S$ the normalization map and set $\nu(x_i)=\nu(x_j)=p$. The linear system $|\OO_S(X)|$ embeds $S$ in $\PP^{g+1}$ and $\nu^*(\OO_S(X))=K_C(x_i+x_j)$. Let $\epsilon:S':=\mbox{Bl}_p(S)\rightarrow S$ be the blow-up of $S$ at $p$ and $E\subset S'$ the exceptional divisor. Note that $C$ viewed as an embedded curve in $S'$ belongs to the linear system $|\epsilon^*\OO_S(1)\otimes \OO_{S'}(-2E)|$ and $C\cdot E=x_i+x_j$. Let $Z\subset S'$ the reduced $0$-dimensional scheme consisting of marked points of $C$ with support $\{x_i, x_j\}^c$.

Since $h^0(C, \OO_C(x_1+\cdots+x_{g}))=2$, we find  that $Z$ together with the tangent plane
$\mathbb{T}_p(X)=\mathbb{T}_p(S)$ span a $(g-1)$-dimensional linear space $\Lambda\subset \PP^{g+1}$. We obtain a $1$-dimensional family in $\dd_g$ by taking the normalization of the intersection curves on $S$ with hyperplanes $H\in (\PP^{g+1})^{\vee}$ passing through $\Lambda$. Equivalently, we note that $$h^0(S', \mathcal{I}_{Z/S'}(C))=h^0(S', \OO_{S'})+h^0(C,K_C(-x_1-\cdots -x_{g}))=2,$$ that is, $|\mathcal{I}_{Z/S'}(C)|$  is a pencil of curves on $S'$. We denote by $\tilde{\epsilon}:\tilde{S}:=\mbox{Bl}_{2g-4}(S')\rightarrow S'$ the blow-up of $S'$ at the $(\epsilon^*(H)-2E)^2=2g-4$ base points of $|\mathcal{I}_{Z/S'}(C)|$, by $f:\tilde{S}\rightarrow \PP^1$ the induced fibration with $(g-2)$ sections
corresponding to the points of $Z$, as well as with a $2$-section given by the divisor $E:=\tilde{\epsilon}^{-1}(E)$. Since $\mbox{deg}(f_{E})=2$, there are precisely two fibres of $f$, say $C_1$ and $C_2$, which are tangent to $E$.  We make a base change or order $2$ via the morphism $f_{E}:E\rightarrow \PP^1$, and consider the fibration
 $$q':Y':=\tilde{S}\times_{\PP^1} E\rightarrow E.$$
Thus $p:Y'\rightarrow \tilde{S}$ is the double cover branched along $C_1+C_2$. Clearly $q'$ admits two sections $E_1, E_2\subset Y'$ such that $p^*(E)=E_1+E_2$ and $E_1\cdot E_2=2$. By direct calculation, it follows that $E_1^2=E_2^2=-3$. To separate the sections
 $E_1$ and $E_2$, we blow-up the two points of intersection $E_1\cap E_2$ and we denote by $q:Y:=\mathrm{Bl}_2(Y')\rightarrow E$  the resulting fibration, which possesses everywhere distinct sections $\sigma_i:E\rightarrow Y'$ for $1\leq i\leq g$, given by the proper transforms of $E_1$ and $E_2$ as well as  the proper transforms of the exceptional divisors corresponding to the points in $Z$.
The numerical characters of the family $\Gamma_{ij}:=\{[q^{-1}(t), \sigma_1(t), \ldots, \sigma_g(t)]: t\in E\} \subset \mm_{g, g}$ are computed as follows:
$$\Gamma_{ij}\cdot \lambda=2(g+1),\ \mbox{ }  \Gamma_{ij}\cdot \delta_{\mathrm{irr}}=2(6g+17), \  \Gamma_{ij}\cdot \psi_l=2 \mbox{ for }  l\in \{i, j\}^c,$$
$$ \Gamma_{ij}\cdot \psi_i=\Gamma_{ij}\cdot \psi_j=-(E_i^2)_{Y'}+2=5, \ \Gamma_{ij}\cdot \delta_{0: \{i, j\}}=2,  \mbox{ } \Gamma_{ij}\cdot \delta_{l: T}=0 \mbox{ for } l\geq 0,  T\subset \{i, j\}^c.$$
We take the $\mathfrak S_g$-orbit of the $1$-cycle $\Gamma_{ij}$ with respect to permuting the marked points,
$$\Gamma:=\frac{1}{g(g-1)} \sum_{i<j} \Gamma_{ij}\in NE_1(\mm_{g, g}),$$
and note that $\Gamma\cdot \dd_{g}=-1$. Each component $\Gamma_{ij}$ fills-up $\dd_g$, which finishes the proof.
\end{proof}

We keep all the notation from the proof of Proposition \ref{extrem} and set $\widetilde{R}:=\pi_*(R)$ and $\widetilde{\Gamma}:=\pi_*(\Gamma) \in NE_1(\cc_{g, g})$. Note that $\widetilde{\Gamma}=\pi_*(\Gamma_{ij})/2$ for all $1\leq i<j\leq g$.
\begin{corollary}\label{numericalsym}
The following intersection identities on $\cc_{g, g}$ hold true:
$$\widetilde{R}\cdot \widetilde{\lambda}=g+1,\ \  \widetilde{R}\cdot \widetilde{\delta}_{\mathrm{irr}}=6g+18,\ \widetilde{R}\cdot \widetilde{\psi}=g \ \ \mbox{ and }\ \  \widetilde{R}\cdot \widetilde{\delta}_{i: c}=0 \ \mbox{ for all pairs }\  (i, c),$$
$$\widetilde{\Gamma}\cdot \widetilde{\lambda}=g+1,\ \widetilde{\Gamma}\cdot \widetilde{\delta}_{\mathrm{irr}}=6g+17,\ \widetilde{\Gamma}\cdot \widetilde{\psi}=g+1 \   \mbox{ and } \  \widetilde{\Gamma}\cdot \widetilde{\delta}_{0: 2}=1, \widetilde{\Gamma}\cdot \widetilde{\delta}_{i: c}=0 \mbox{ for } \ (i, c)\neq (0, 2).$$
It follows that $\widetilde{R}\cdot K_{\cc_{g, g}}=2g-23$ and $\widetilde{\Gamma}\cdot K_{\cc_{g, g}}=2g-21$.
\end{corollary}

\section{The Mukai model of $\mm_{g, g}$}
Having showed that the divisor $\dd_g\in \mathrm{Eff}(\mm_{g, g})$ is extremal when $g\leq 11$, our next aim is to construct a "modular" birational contraction of $\mm_{g, g}$, such that $\dd_g$ appears among its exceptional divisors. We achieve this goal for $7\leq g\leq 9$, using  Mukai's fundamental work on classification of Fano varieties. We recall that for $6\leq g\leq 9$, there exists a $n_g$-dimensional Fano variety $V_g\subset \PP^{N_g}$ of index $n_g-2$ and $\rho(V_g)=1$, where $N_g:=g+n_g-2$, such that general $1$-dimensional complete intersections of $V_g$ are canonical curves $[C]\in \cM_g$ with general moduli.  One has the following table, see \cite{M1} or \cite{M3} p. 256:
\begin{center}
\begin{tabular}{c|c|c|cc}
$g$ & $n_g$ & $N_g$& $V_g$ &
\\
\hline
$6$ & $5$ & $9$ & {\Small{Quadric section of}} \ $G(2, 5)$
\\
\hline $7$ & $10$ & $15$ & {\Small{Spinor variety}} \ $OG(5, 10)$ \\
\hline $8$ & $8$ & $14$ & {\Small{Grassmannian}} \ $G(2, 6)$ \\
\hline $9$ &  $6$ & $13$ & {\Small{Symplectic Grassmannian}} \ $SG(3, 6)$\\
\end{tabular}
\end{center}
The automorphism group $\mbox{Aut}(V_g)$ acts in a natural way on the product $V_g^{g}$ and we choose the polarization $\L:=\OO_{V_g}(1)\boxtimes \cdots \boxtimes \OO_{V_g}(1)\in \mbox{Pic}(V_g^g)$. For $7\leq g\leq 9$, we call the GIT-quotient
$$
\mathfrak{M}_{g, g} := (V_g^g)^{\mathrm{ss}}(\L) \dblq \mbox{Aut}(V_g)
$$ the \emph{Mukai model} of $\mm_{g, g}$. For $g=6$, it is not clear that $\mbox{Aut}(V_6)$ is a reductive group and leave the question of the nature of  $\mathfrak{M}_{6, 6}$ aside for further study (We are grateful to the referee for pointing this issue out to us). When $7\leq g\leq 9$, there exists a birational rational map
$$f_g:\mm_{g, g}\dashrightarrow \mathfrak{M}_{g, g}, \  \ \ \ f_g\bigl([C, x_1, \ldots, x_g]\bigr):=(x_1, \ldots, x_g) \ \ \mathrm{mod}\ \mathrm{Aut}(V_g).$$
The inverse map is given by $f_g^{-1}(x_1, \ldots, x_g):=[\langle x_1, \ldots, x_g\rangle \cap V_g, x_1, \ldots, x_g]$. It is not hard to see that $f_g$ contracts all boundary divisors $\Delta_{i: T}$, where $i\geq 1$. The point is that a stable curve with a disconnecting node cannot appear as a linear section of $V_g$, that is, stable curves from the divisors $\Delta_i\subset \mm_g$ where $i>0$, correspond to non-nodal curvilinear sections of $V_g$. Accordingly, the locus of planes $\Lambda\in G(g, N_g+1)$ such that $\Lambda\cap V_g$ is not an irreducible curve with at worst nodal singularities, has codimension at least $2$ in $G(g, N_g+1)$ (see also \cite{FV} Proposition 4.2).  From \cite{M1}, \cite{M4}, it also follows that $f_g$ also blows-down the pull-back of the unique Brill-Noether divisor on $\mm_g$ when $g\neq 4, 6$ (respectively the \emph{Petri divisor} on $\mm_4$ and $\mm_6$). By comparing Picard numbers, the exceptional divisor $\mbox{Exc}(f_g)$ must contain one extra component:
\begin{proposition}\label{contraction}
For $7\leq g\leq 9$, the rational morphism $f_g$ contracts the divisor $\dd_g$.
\end{proposition}
\begin{proof} It suffices to note that $f_g$ blows-down the covering curves $R\subset \dd_g\subset \mm_{g, g}$ constructed in the course of proving Proposition \ref{extrem}.
\end{proof}
We use the existence of the Mukai variety $V_g$, to establish Theorem \ref{symprod}.
\vskip 2pt
\noindent
\emph{Proof of Theorem \ref{symprod}.} In the range $6\leq g\leq 9, n\leq g$, the unirationality of $\cc_{g, n}$ follows from that of $\mm_{g, n}$. Indeed, the parameter space $$\Sigma:=\{\bigl((x_1, \ldots, x_n), \Lambda)\in V_g^n\times G(g, N_g+1): x_i\in \Lambda, \mbox{ for } i=1, \ldots, n\bigr\}$$ maps dominantly onto $\mm_{g, n}$ via the map $\bigl((x_1, \ldots, x_n), \Lambda\bigr)\mapsto [V_g\cap \Lambda, x_1, \ldots, x_n]$.
Since $\Sigma$ is a Grassmann bundle over the rational variety $V_g^n$, the conclusion follows. It is proved in \cite{FP} that $\mm_{11, n}$ (thus $\cc_{g, n}$ as well), is uniruled for $n\leq 10$. Similarly, $\mm_{10, n}$ is uniruled for $n\leq 9$, cf. {\it{loc. cit}}.
$\hfill$ $\Box$
\hfill
\vskip 3pt

 These results can be improved when $g\leq 6$ using plane models of minimal degree:
\begin{proposition}
$\cc_{g, n}$ is unirational for all $g\leq 6$ and $n\geq 0$.
\end{proposition}
\begin{proof} One used the representation of the general curve $[C]\in \cM_g$ as a plane model of degree $d:=[(2g+8)/3]$ with $\delta:={d-1\choose 2}-g$ nodes in general position, cf. \cite{AC2}. We describe the details for the case $g=d=6$. We choose general points $p_1, \ldots, p_4\in \PP^2$ and fix an integer $l\geq 3$ such that $6l-20<n\leq 6l-14$. On the normalization $C$ of a curve $\Gamma\in |\OO_{\PP^2}(6)(-2\sum_{i=1}^4 p_i)|$, the complete linear series $K_C(l-3)=\mathfrak g_{6l-8}^{6l-14}$ is cut out by degree $l$ curves passing through $p_1, \ldots, p_4$.
We define the incidence correspondence
$$\cU:=\Bigl\{\bigl(\Gamma, X_l, \{a_i\}_{i=1}^{6l-8-n}\bigr)\in |\OO_{\PP^2}(6)(-2\sum_{i=1}^4 p_i)| \times |\OO_{\PP^2}(l)(-\sum_{i=1}^4 p_i)|\times (\PP^2)^{6l-8-n}: $$
$$\Gamma\cdot X_l\geq 2(p_1+\cdots+p_4)+a_1+\cdots+a_{6l-8-n}\Bigr\}.$$
We note that $\cU$ is rational. The residuation map $r:\cU\dashrightarrow \cc_{6, n}$ defined by
$$r\bigl(\Gamma, X_l, \{a_i\}_{i=1}^{6l-8-n}\bigr):=[C, D], \ \  \mbox{ where } \ \Gamma\cdot X_l=2\sum_{i=1}^4 p_i+\sum_{i=1}^{6l-8-n} a_i+D, $$
and $C\rightarrow \Gamma$ is the normalization map, is dominant. Thus $\cc_{6, n}$ is unirational.
\end{proof}

\section{The Kodaira dimension of $\mm_{11, 11}$}
On $\mm_{11}$ there exist two divisors of Brill-Noether type consisting of curves with special linear series,  namely the closure of the locus of $6$-gonal curves   $$\cM_{11, 6}^1:=\{[C]\in \cM_{11}: G^1_6(C)\neq \emptyset\}$$
and the closure of the locus $\cM_{11, 9}^2:=\{[C]\in \cM_{11}: G^2_9(C)\neq \emptyset\}$.  The divisors $\mm_{11, 6}^1$ and $\mm_{11, 9}^2$ are irreducible, distinct, and their classes  are proportional, cf. \cite{EH2}. Precisely, there  are explicit constants $c_{11, 1, 6}, c_{11, 2, 9}\in \mathbb Z_{>0}$, such that $$\mathfrak{bn}_{11}:\equiv \frac{1}{c_{11, 1, 6}}\ \mm_{11, 6}^1\equiv \frac{1}{c_{11, 2,  9}}\ \mm_{11, 9}^2\equiv 7\lambda-\delta_0-5\delta_1-9\delta_2-12\delta_3-14\delta_4-15\delta_5\in \mbox{Pic}(\mm_{11}).$$
By interpolating, we find the following explicit canonical divisor:
\begin{equation}\label{canrep}
K_{\mm_{11, 11}}\equiv \dd_{11}+ 2\cdot  \phi^*(\mathfrak{bn}_{11})+\sum_{i=0}^5\sum_{c=0}^{11} d_{i: c}\ \delta_{i: c},
\end{equation}
where $$d_{0: c}=\frac{c^2+c-4}{2}\ \ \mbox{  } \mbox{ for }c\geq 2,\ \  \ d_{1: c}=8+{|c-1|+1\choose 2} \ \mbox{ for } c\geq 1,$$
 $$ d_{1: 0}=7, \ \ d_{2: c}=16+{|c-2|+1\choose 2},\  \ d_{3: c}=22+{|c-3|+1\choose 2},$$
 $$  d_{4: c}=26+{|c-4|+1\choose 2}, \ \mbox{ and } \ \   d_{5: c}=28+{|c-5|+1\choose 2}.$$
 Similarly, at the level of the universal symmetric product $\cc_{11, 11}$ one has the relation
 \begin{equation}\label{symcanrep}
 K_{\cc_{11, 11}}\equiv \widetilde{\mathcal{D}}_{11}+2\cdot \varphi^*(\mathfrak{bn}_{11})+\sum_{(i, c)\neq (0, 2)} d_{i: c}\ \widetilde{\delta}_{i: c}.
 \end{equation}

One already knows that multiples of $\dd_{11}$ are non-moving divisors on $\mm_{11, 11}$. We show that $\dd_{11}$ does not move in any multiple of the canonical linear system on $\mm_{11, 11}$.
\begin{proposition}\label{eliminate}
For each integer $n\geq 1$, one has an isomorphism
$$H^0\bigl(\mm_{11, 11}, \OO_{\mm_{11, 11}}(nK_{\mm_{11, 11}})\bigr)\cong H^0\bigl(\mm_{11, 11}, \OO_{\mm_{11, 11}}(nK_{\mm_{11, 11}}-n\dd_{11})\bigr).$$ In particular, $\kappa\bigl(\mm_{11, 11}\bigr)=\kappa\bigl(\mm_{11, 11}, K_{\mm_{11, 11}}-\dd_{11}\bigr)$. Furthermore, on $\cc_{11, 11}$, one has that $\kappa \bigl(\cc_{11, 11}\bigr)=\kappa\bigl(\cc_{11, 11}, K_{\cc_{11, 11}}-\widetilde{\mathcal{D}}_{11}\bigr)$.
\end{proposition}
\begin{proof} Using the notation and results from Proposition \ref{extrem}, we recall that we have constructed a curve $R\subset \mm_{11, 11}$ moving in a family which  fills-up the divisor $\dd_{11}$, such that  $R\cdot \dd_{11}=-1$ and $R\cdot \delta_{i: S}=0$, for all $i\geq 0$ and $T\subset \{1, \ldots, g\}$.
All points in $R$ correspond to nodal curves lying  on a fixed $K3$ surface $S$, which by the generality assumptions, can be chosen such that
$\mbox{Pic}(S)=\mathbb Z$. Applying \cite{Laz}, all underlying genus $11$ curves corresponding to points in $R$ satisfy the Brill-Noether theorem, in particular $R\cdot \phi^*(\mathfrak{bn}_{11})=0$, that is, $R\cdot K_{\mm_{11, 11}}=R\cdot \dd_{11}=-1$. It follows that for any effective divisor $E$ on $\mm_{11, 11}$ such that $E\equiv nK_{\mm_{11, 11}}$, one has that $R\cdot E=-n$. Moreover, the class $E-n\dd_{11}$ is still effective and then $|nK_{\mm_{11, 11}}|=n\dd_{11}+|nK_{\mm_{11, 11}}-n\dd_{11}|$. The proof in the case of $\cc_{11, 11}$ is similar. One uses that $\pi^*(\widetilde{\mathcal{D}}_{11})=\dd_{11}$, hence $\widetilde{R}\cdot \widetilde{\mathcal{D}}_{11}=R\cdot \dd_{11}=-1$, as well as $\widetilde{R}\cdot K_{\cc_{11}}=-1$. The rest of the argument is identical.
\end{proof}
We are in a position to complete the proof of Theorem \ref{gen11}:
\begin{theorem} We have that $$\kappa\bigl(\mm_{11, 11}, \ 2\cdot \phi^*(\mathfrak{bn}_{11})+\sum_{i, c} d_{i: c}\cdot\delta_{i: c}\bigr)=
\kappa\bigl(\cc_{11, 11}, \ 2\cdot \varphi^*(\mathfrak{bn}_{11})+\sum_{(i, c)\neq (0, 2)} d_{i: c}\cdot\widetilde{\delta}_{i: c}\bigr)
=19.$$
It follows that the Kodaira dimension of both $\mm_{11, 11}$ and $\cc_{11, 11}$ equals $19$.
\end{theorem}
\begin{proof} To simplify the proof, we define a few divisors classes on $\mm_{11, 11}$:
 $$A:=2\cdot \phi^*(\mathfrak{bn}_{11})+\sum_{i\geq 0, c} d_{i: c}\ \delta_{i: c}\equiv K_{\mm_{11, 11}}-\dd_{11}\ \mbox{ and } \ A':=A-\sum_{c=2}^{11} d_{0: c}\
 \delta_{0: c},$$
 as well as, $B:= \mathfrak{bn}_{11}+4\delta_3+7\delta_4+8\delta_5\in \mbox{Pic}(\mm_{11})$.

We claim that for all integers $n\geq 1$, one has isomorphisms,
$$H^0\bigl(\mm_{11, 11}, \OO_{\mm_{11, 11}}(nA)\bigr)\cong H^0\bigl(\mm_{11, 11}, \OO_{\mm_{11, 11}}(nA')\bigr).$$
Indeed, we fix a set of labels $T\subset \{1, \ldots, 11 \}$ such that $\#(T)\geq 2$ and consider a pencil
$$\bigl\{[C_t, x_i(t), p(t): i\in T^c]\bigr\}_{t\in \PP^1} \subset \mm_{11, 12-\#(T)}, $$ of $(12-\#(T))$-pointed curves of genus $11$ on a general $K3$ surface $S$, with marked points being labeled by elements in $T^c$ as well by another label $p(t)$. The pencil is induced by a fibration obtained from a Lefschetz pencil of genus $11$ curves on $S$, with regular sections given by $(12-\#(T))$ of the exceptional divisors obtained by blowing-up $S$ at the $(2g-2)$ base points of the pencil. To each element in this pencil, we attach at the marked point labeled by $p(t)$, a fixed copy of $\PP^1$ together with fixed marked points $x_i\in \PP^1-\{\infty\}$, for $i\in T$. The gluing identifies the point $p(t)\in C_t$ with $\infty\in \PP^1$. If $R_T\subset \mm_{11, 11}$ denotes the resulting family,  we compute:
$$R_T\cdot \lambda=g+1,\ R_T\cdot \delta_{\mathrm{irr}}=6(g+3), \ \ R_T\cdot \delta_{0: T}=-1,\ R_T\cdot \psi_i=1 \mbox{ for }i\in T^c, \  \ R_T\cdot \psi_i=0\mbox{ for }i\in T.$$ Moreover, $R_T$ is disjoint from all remaining boundary divisors of $\mm_{11, 11}$. One finds that
$R_T\cdot \phi^*(\mathfrak{bn}_{11})=0$.  Thus for any effective divisor $E\subset \mm_{11, 11}$ such that $E\equiv nA$, we find that $R_T\cdot E=-n d_{0, c}$.

Since for all $T$, the pencil $R_T$  fills-up the divisor $\Delta_{0: T}$, we can deform the curves $R_T\subset \Delta_{0: T}$, to find that $E-\sum_{c=2}^{11}nd_{0: c} \cdot \delta_{0: c}$
is still an effective class, that is,
$$|nA|=\sum_{c=2}^{11} nd_{0: c}\cdot \Delta_{0: c}+|nA'|,$$
which proves the claim.
Next, by direct calculation we observe that the class $A'-2\phi^*(B)$ is effective.  Zariski's Main Theorem gives that  $\phi_*\phi^*\OO_{\mm_{11}}(B)=\OO_{\mm_{11}}(B)$, thus
$$\kappa\bigl(\mm_{11, 11}, A'\bigr)\geq \kappa\bigl(\mm_{11, 11}, \phi^*(B)\bigr)=\kappa(\mm_{11}, B)=19.$$
The last equality comes from \cite{FP} Proposition 6.2: The class $B$ contains the pull-back of an ample class under the Mukai map \cite{M2} $$q_{11}: \mm_{11, 11}\dashrightarrow \ff_{11},\ \  \ [C, x_1, \ldots, x_{11}]\mapsto [S\supset C,\  \OO_S(C)],$$
to a compactification of the moduli space of polarized $K3$ surfaces of degree $20$.

On the other hand, since $\phi^*(\delta_i)=\sum_{S} \delta_{i: S}$ for $1\leq i\leq 5$, there is a divisor class on $\mm_{11}$ of type
$B':= 2\cdot \mathfrak{bn}_{11}+\sum_{i=1}^5 a_i \delta_i\in \mathrm{Pic}(\mm_{11})$,
with $a_i\geq 0$, such that $\phi^*(B')-A'$ is an effective divisor. It follows that
$$\kappa\bigl(\mm_{11, 11}, A'\bigr)\leq \kappa\bigl(\mm_{11, 11}, \phi^*(B')\bigr)=\kappa(\mm_{11}, B').$$
If $R_{11} \subset \mm_{11}$ is the family corresponding to a Lefschetz pencil of curves of genus $11$ on a fixed $K3$ surface, then $R_{11}\cdot B'=0$. The pencil $R_{11}$ moves in a $11$-dimensional family inside $\mm_{11}$ which is contracted to a point by any linear series $|nB'|$ on $\mm_{11}$ with $n\geq 1$ (in fact a general curve $R_{11}$ is \emph{disjoint} from the base locus of $|nB'|$). One
finds that $\kappa(\mm_{11}, B')\leq 19$, which completes the proof. The case of  $\cc_{11, 11}$ proceeds with obvious modifications.
\end{proof}
\section{The Kodaira dimension of $\cc_{10, 10}$}
The geometry of $\mm_{10}$ is governed to a large extent by the divisor $\kk_{10}$ of curves lying on $K3$ surfaces. It is shown in \cite{FP} that $\kk_{10}$ is an irreducible divisor of class $$\kk_{10}\equiv 7\lambda-\delta_0-5\delta_1-9\delta_2-12\delta_3-14\delta_4-b_5\delta_5\in \mathrm{Pic}(\mm_{10}),$$
where $b_5\geq 6$. Furthermore, $s(\kk_{10})=7$ is the minimal slope of an effective divisor on $\mm_{10}$.
The irreducible divisor $\kk_{10}$ is rigid. Indeed, if $R_{10}\subset \kk_{10}$ is the covering family obtained by blowing-up the base points of a pencil of curves of genus $10$ on a fixed $K3$ surface, then $R_{10}\cdot \kk_{10}=-1$ (see \cite{FP} Lemma 2.1). Since $R_{10}\cdot \delta_i=0$ for $1\leq i\leq 5$, this argument proves that any divisor $D\in \mathrm{Eff}(\mm_{10})$ with $s(D)=s(\kk_{10})$, is rigid as well.

Comparing the expression of $[\kk_{10}]$ with that of $K_{\cc_{10, 10}}$, we obtain the formula
\begin{equation}\label{canrep10}
K_{\cc_{10, 10}}\equiv \widetilde{\mathcal{D}}_{10}+ 2\cdot  \varphi^*(\kk_{10})+\sum_{(i, c)\neq (0, 2)} d_{i: c}\ \widetilde{\delta}_{i: c},
\end{equation}
where remarkably, the coefficients $d_{i: c}$ have exactly the same values as in formula (\ref{canrep10}) for $1\leq i\leq 4$, while
$d_{5: c}=2b_5-2+{|c-5|+1\choose 2}$, for $0\leq c\leq 10$. We point out that the coefficient $d_{0: 2}$ of $\widetilde{\delta}_{0: 2}$ in formula (\ref{canrep10}), equals $0$.

\begin{theorem}
For each $n\geq 1$,  there is an isomorphism of groups
$$H^0\bigl(\cc_{10, 10}, \OO_{\cc_{10, 10}}(nK_{\cc_{10, 10}})\bigr)\cong H^0\bigl(\cc_{10, 10}, \OO_{\cc_{10, 10}}(nK_{\cc_{10, 10}}-n\widetilde{\mathcal{D}}_{10})\bigr).$$
\end{theorem}
\begin{proof}
We use Corollary \ref{numericalsym}. Through a general point of the divisor  $\dd_{10}$ on $\mm_{10, 10}$ there passes a curve
$\widetilde{\Gamma} \subset \cc_{10, 10}$ such that
$\widetilde{\Gamma} \cdot \widetilde{\mathcal{D}}_{10}=\widetilde{\Gamma}\cdot K_{\cc_{10, 10}}=-1$ and $\Gamma \cdot \widetilde{\delta}_{i: c}=0$ for $(i, c)\neq (0, 2)$. One obtains that $|nK_{\cc_{10, 10}}|=n\widetilde{D}_{10}+|nK_{\cc_{10, 10}}-n\widetilde{D}_{10}|$.
\end{proof}

\vskip 3pt
\noindent
\emph{End of proof of Theorem \ref{picard} when $g=10$.}
We define the following divisor classes on $\cc_{10, 10}$:
$$A:=2\cdot \varphi^*(\kk_{10})+\sum_{(i, c)\neq (0, 2)} d_{i: c}\ \widetilde{\delta}_{i: c}\equiv K_{\cc_{10, 10}}-\widetilde{\mathcal{D}}_{10}\ \ \mbox{ and } \ \ A':=A-\sum_{c=3}^{10} d_{0: c}\ \widetilde{\delta}_{0: c}.$$
\noindent
We claim that $H^0\bigl(\cc_{10, 10}, \OO_{\cc_{10, 10}}(nA)\bigr)\cong H^0\bigl(\cc_{10, 10}, \OO_{\cc_{10, 10}}(nA')\bigr)$ for all $n\geq 1$.  Indeed, we fix a set of labels $T\subset \{1, \ldots, 10\}$ with $c:=\#(T)\geq 3$, as well as two indexes $i, j\in T^c$ and consider the $1$-cycle $\Gamma_{ij}\subset \mm_{10, 11-c}$ constructed in Proposition \ref{extrem}. We label by $\{p(t), x_l(t): t\in \PP^1\}_{l\in T^c}$ the sections of the family. We obtain a covering curve $\Gamma'_{ij}$ for the divisor $\Delta_{0: T}\subset \mm_{10, 10}$, by attaching along the section $p(t)$ a fixed $(c+1)$-pointed rational curve to each of the curves in $\Gamma_{ij}$, in a way that the marked points labeled by $T$ are precisely those lying on the rational component. Then $\widetilde{\Gamma}_{0:c}:=\pi_*(\Gamma'_{ij})\subset \cc_{10, 10}$ is a covering curve for $\widetilde{\Delta}_{0: c}$. From Corollary \ref{numericalsym}, $\widetilde{\Gamma}_{0: c}\cdot \widetilde{\delta}_{0: c}<0$ and $\widetilde{\Gamma}_{0: c}$ has intersection number $0$ with all the components of $\mbox{supp}(A)-\widetilde{\Delta}_{0: c}$ (Note that $\widetilde{\Delta}_{0: 2}$ does not appear among these components). We repeat this argument for all divisors $\widetilde{\Delta}_{0: c}$, where $3\leq c\leq g$, and the claim becomes obvious. We finish the proof by using the same argument as at the end of the proof of Theorem \ref{eliminate}: There exists an effective class $B'\in \mathrm{Eff}(\mm_{10})$ such that $B'-\kk_{10}$ is effective, $s(B')=s(\kk_{10})=7$, and such that $\varphi^*(B')-A'$ is effective. Then $\kappa(\cc_{10, 10})=\kappa(\cc_{10, 10}, A')\leq \kappa(\cc_{10, 10}, \varphi^*(B'))=\kappa(\mm_{10}, B')=0$.
$\hfill$ $\Box$

\section{The uniruledness of $\mm_{g, n}$}

We formulate two general principles which we use in proving the uniruledness of some moduli spaces $\mm_{g, n}$. We begin with the following  consequence of \cite{BDPP}:
\begin{proposition}\label{uniruled0}
Let $X$ be a normal projective $\mathbb Q$-factorial variety and $D\subset X$ an irreducible divisor filled-up by curves $\Gamma\subset X$, such that $\Gamma\cdot D\geq 0$ and $\Gamma\cdot K_X<0$. Then $X$ is uniruled.
\end{proposition}
\begin{proof}
First note that $K_X$ is not pseudo-effective. Assume that on the contrary, $K_{X}$ lies in the closure $\overline{\mbox{Eff}}(X)$ of the effective cone of $X$, set $\alpha:=\mbox{sup}\{t\in \mathbb Q_{\geq 0}: K_X-t D\in \overline{\mbox{Eff}}(X)\}$, and write   $A:=K_{X}-\alpha  D\in \overline{\mbox{Eff}}(X)$. The curve $\Gamma$ deforms in a family filling-up $D$, hence because of the maximality of $\alpha$, one has $\Gamma\cdot A\geq 0$ and $\Gamma \cdot K_X=\alpha \Gamma \cdot D+\Gamma \cdot A\geq 0$, which is a contradiction.
Thus $K_X\notin \overline{\mbox{Eff}}(X)$ and we claim that this implies that $X$ is uniruled. Indeed, let  us consider a resolution $\mu:X'\rightarrow X$ of $X$. It is enough to prove that $X'$ is uniruled and according to \cite{BDPP} this is equivalent to showing that $K_{X'}$ is not pseudo-effective. But $K_{X'}\equiv \mu^*(K_X)+\sum_{i} a_i E_i$, where $a_i\in \mathbb Q$ and the divisors $E_i\subset X'$ are the components of the exceptional locus $\mbox{Exc}(\mu)$. Since the divisors $E_i$ are $\mu$-exceptional and $K_X$ is not pseudo-effective, we find that no divisor of the form $\mu^*(K_X)+\sum_i a_i' E_i$, where $a_i'\geq 0$, can be pseudo-effective either. In particular, $K_{X'}$ is not pseudo-effective, hence $X'$ is uniruled.
\end{proof}


We can extend this principle to the case of several divisors as follows:
\begin{proposition}\label{uniruled}
Let $X$ be a normal projective $\mathbb Q$-factorial variety and suppose $D_1, D_2\subset X$ are irreducible effective $\mathbb Q$-divisors such that there exist covering curves
$\Gamma_i\subset D_i$, with $\Gamma_i\cdot D_i<0$ for $i=1, 2$ (in particular both $D_i\in \mathrm{Eff}(X)$ are non-movable divisors). Assume furthermore that
\begin{equation}
\begin{vmatrix}
\Gamma_1\cdot D_1& \Gamma_1\cdot D_2\\
\Gamma_2\cdot D_1& \Gamma_2\cdot D_2\\
\end{vmatrix} \leq 0,\mbox{ }\mbox{ }\
\begin{vmatrix}
\Gamma_1\cdot K_X& \Gamma_1\cdot D_1\\
\Gamma_2\cdot K_X& \Gamma_2\cdot D_1\\
\end{vmatrix}<0.
\end{equation}
Then $X$ is uniruled.
\end{proposition}
\begin{proof} According to \cite{BDPP}, it suffices to prove that $K_X$ is not pseudo-effective. By contradiction, we choose $\alpha, \beta\in \mathbb R_{\geq 0}$ maximal such that $K_X-\alpha D_1-\beta D_2\in \overline{\mathrm{Eff}}(X)$. Then we can write down the inequalities
$$\Gamma_1 \cdot K_X\geq \alpha (\Gamma_1\cdot D_1) +\beta (\Gamma_1\cdot D_2) \ \mbox{ and } \Gamma_2\cdot K_X\geq \alpha (\Gamma_2\cdot D_1)+\beta (\Gamma_2\cdot  D_2).$$
Eliminating $\alpha$, the resulting inequality contradicts the assumption $\beta\geq 0$.
\end{proof}
We turn our attention to the proof of Theorem \ref{genul8} which we split in three parts:

\begin{theorem}\label{gen5} $\mm_{5, n}$ is uniruled for $n\leq 13$.
\end{theorem}
\begin{proof}
A general $2$-pointed curve $[C, x, y]\in \cM_{5, 2}$ carries a finite number of linear series $L\in W^2_6(C)$, such that
if $\nu_L: C\stackrel{|L|}\longrightarrow \Gamma\subset \PP^2$ is the induced plane model, then $\nu_L(x)=\nu_L(y)=p_1$. Note that
$\Gamma$ has nodes, say $p_1, \ldots, p_{5}$, and $\mbox{dim }|\OO_{\PP^2}(\Gamma)(-2\sum_{i=1}^5 p_i)|=12$.

We pick general points $\{x_i\}_{i=1}^{11}$ and $\{p_j\}_{j=1}^5 \subset \PP^2$, then consider the pencil of sextics passing with multiplicity $1$ through $x_1, \ldots, x_{11}$ and having nodes (only) at $p_1, \ldots, p_5$. The pencil induces a fibration $f':S\rightarrow \PP^1$, where $S:=\mathrm{Bl}_{21}(\PP^2)$ is obtained from $\PP^2$ by blowing-up $p_1, \ldots, p_5$, $x_1, \ldots, x_{11}$, as well as the remaining unassigned base points of the pencil. The exceptional divisors $E_{x_i}\subset S$ provide $11$ sections of $f'$. The exceptional divisor $E_{p_1}$ induces a $2$-section. Making a base change via the map $f'_{E_{p_1}}: E_{p_1}\rightarrow \PP^1$,  the $2$-section $E_{p_1}$ splits into two sections $E_x$ and $E_y$  meeting at $2$ points. Blowing these points up, we arrive at a
fibration $f:Y\rightarrow E_{p_1}$, carrying $13$ everywhere disjoint sections, $\tilde{E}_x, \tilde{E}_y, \tilde{E}_{x_1}, \ldots, \tilde{E}_{x_{11}}$, where $\tilde{E}_{x_i}\subset Y$ denotes the inverse image of $E_{x_i}$, and $\tilde{E}_x, \tilde{E}_y$ denote the proper transforms of $E_x$ and $E_y$ respectively. This induces a family of pointed stable curves
$$\Gamma:=\bigl\{[C_{\lambda}:=f^{-1}(\lambda),\  \tilde{E}_x\cdot C_{\lambda},\  \tilde{E}_y\cdot C_{\lambda},\  \ \tilde{E}_{x_1}\cdot C_{\lambda}, \ldots, \ \tilde{E}_{x_{11}}\cdot C_{\lambda}]:
\lambda\in E_{p_1}\bigr\}\subset \mm_{5, 13}.$$

We compute the numerical characters of $\Gamma$ (see also the proof of Proposition \ref{extrem}):
$$\Gamma\cdot \lambda=\mathrm{deg}(f_{E_{p_1}})\bigl(\chi(S, \OO_S)+g-1\bigr)=10,\ \Gamma\cdot \delta_{\mathrm{irr}}=\mathrm{deg}(f_{E_{p_1}})\bigl(c_2(S)+4g-4\bigr)=80,$$
$$\Gamma\cdot \psi_x=\Gamma\cdot \psi_y=5, \ \ \Gamma\cdot \psi_{x_1}=\cdots =\Gamma\cdot \psi_{x_{11}}=2, \ \ \Gamma\cdot \delta_{0: xy}=2,$$
whereas $\Gamma$ is disjoint from the remaining boundary divisors.
One finds, $\Gamma\cdot K_{\mm_{5, 13}}=-2$, which completes the proof.
\end{proof}
\begin{remark} It is known that $\mm_{5, 15}$ is of general type, \cite{F1} p. 865. Using the fact that $12$ general points in $\PP^4$ determine a canonical curve of genus $5$, it is proved in \cite{CF} that $\mm_{5, n}$ is rational when $n\leq 12$. Hence Theorem \ref{gen5} settles the cases  $\mm_{5, 13}$.
\end{remark}

\begin{theorem}\label{gen78} $\mm_{8, n}$ is uniruled for $n\leq 12$.
\end{theorem}
\begin{proof} We apply Proposition \ref{uniruled} when $D_1$ is a suitable multiple of the Brill-Noether divisor on $\mm_8$ consisting of curves with a $\mathfrak g^2_7$, that is,
 $$D_1\equiv 2\cdot \mathfrak{bn}_8:=\frac{2}{c_{8, 2, 7}} \mm_{8, 7}^2\equiv 22\lambda-3\delta_0-14\delta_1-24\delta_2-30\delta_3-32\delta_4\in \mbox{Pic}(\mm_8).$$
 We also set $D_2:=\Delta_{\mathrm{irr}}\in \mathrm{Eff}(\mm_{8, n})$. To construct a covering curve $\Gamma_1\subset D_1$, we lift to $\mm_{8, n}$ a Lefschetz pencil of $7$-nodal plane septics.  The fibration $f:\mathrm{Bl}_{28}(\PP^2)\rightarrow \PP^1$ obtained by blowing-up the $21+7$ base points of a general pencil of $7$-nodal plane septics, induces a covering curve $m:\PP^1\rightarrow \mm_8$ for the irreducible divisor $\mm_{8, 7}^2$. The numerical invariants of this pencil are
$$m^*(\lambda)=\chi(S, \OO_S)+g-1=8\ \mbox{ and } m^*(\delta_0)=c_2(S)+4(g-1)=59,$$
while  $m^*(\delta_i)=0$ for $i=1, \ldots, 4$. Moreover, for $n$ as above, $f$ carries $n$ sections given by the exceptional divisors corresponding to $n$ of the unassigned base points. If $\Gamma_1\subset \mm_{8, n}$ denotes the resulting, then
$$\Gamma_1\cdot \lambda=\phi_*(\Gamma_1) \cdot \lambda=8, \ \Gamma_1\cdot \delta_{\mathrm{irr}}=\phi_*(\Gamma_1) \cdot \delta_{\mathrm{irr}}=59,  \ \Gamma_1\cdot \psi_i=1 \mbox{ for } i=1, \ldots, n,$$
and $\Gamma_1 \cdot \delta_{i: T}=0$. It follows that $\Gamma_1\cdot D_1=-1, \ \Gamma_1\cdot K_{\mm_{8, n}}=n-14$ and $\Gamma_1\cdot D_2=59$.

We construct a covering curve $\Gamma_2\subset D_2$ and start with a general pointed curve $[C, x_1, \ldots, x_{n+1}]\in \mm_{7, n+1}$. We identify $x_{n+1}$ with a moving point $y\in C$, that is, take
$$\Gamma_2:=\bigl\{\bigl[\frac{C}{y\sim x_{n+1}},  x_1, \ldots, x_{n} \bigr]: y\in C\bigr\}\subset \mm_{8, n}.$$
It is easy to compute that
$\Gamma_2\cdot \lambda=0$, \  $\Gamma_2\cdot \delta_{\mathrm{irr}}=-2g(C)=-14$, \ $\Gamma_2\cdot \delta_{1: \emptyset}=1,\ \Gamma_2\cdot \psi_i=1,
\mbox{ for } i=1, \ldots, n,$
and $\Gamma_2\cdot \delta_{i: T}=0$ for $(i, T)\neq (1, \emptyset)$. Therefore $\Gamma_2\cdot D_1=28$ and $\Gamma_2\cdot K_{\mm_{8, n}}=25+n$. The assumptions of Proposition \ref{uniruled} are satisfied when $n\leq 12$.
\end{proof}
\begin{remark} The results of Theorem \ref{gen78} are almost optimal. The space $\mm_{8, 14}$ is of general type, \cite{Log}. Note that it was already known \cite{Log}, \cite{CF}, that
$\mm_{8, n}$ is unirational for $n\leq 11$, thus the improvement here is the case $\mm_{8, 12}$.
\end{remark}
\vskip 3pt
\begin{proposition}\label{gen9}
$\mm_{9, n}$ is uniruled for $n\leq 10$ \ (in fact unirational for $n\leq 9$).
\end{proposition}
\begin{proof} We apply Proposition \ref{uniruled0}, when $D:=\phi^*(\mm_{9, 5}^1)$ is the pull-back of the $5$-gonal locus inside $\mm_9$. If $[C]\in \mm_{9, 5}^1$ is a general $5$-gonal curve and $A\in W^1_5(C)$ is the (unique) $\mathfrak g^1_5$, then there exists an effective divisor $D\in C_3$, such that $h^0(C, A\otimes \OO_C(D))\geq 3$, cf. \cite{AC2}. In particular, $A\otimes \OO_C(D)\in W^2_8(C)$ induces a plane model of $C$ having a $3$-fold point, such that $|A|$ can be retrieved by projecting from this point.

To obtain a covering curve for $D$, we start with general points $p_0, p_1, \ldots, p_9\in \PP^2$ and consider the surface $\epsilon:S:=\mathrm{Bl}_{p_0, \ldots, p_9}(\PP^2) \rightarrow \PP^2$ together with the line bundle $\mathcal{L}:=\epsilon^*\OO_{\PP^2}(8)\otimes \OO_S(-3E_{p_0}-\sum_{i=1}^9{E_{p_i}})\in \mathrm{Pic}(S)$. Note that $\mbox{dim } |\L|=11$.  We fix $10$ general points $x_1, \ldots, x_{10}\in S$, hence the pencil $|\mathcal{I}_{\{x_1, \ldots, x_{10}\}/S}\otimes \L|$ induces a curve $\Gamma\subset \mm_{9, 10}$. Standard calculations yield that $\Gamma\cdot \lambda=9$, $\Gamma\cdot \delta_{\mathrm{irr}}=64$ and $\Gamma\cdot \psi_i=1$ for $i=1, \ldots, 10$. Therefore $\Gamma\cdot K_{\mm_{9, 10}}=-1$, while $\Gamma\cdot \phi^*(\mm_{9, 5}^1)>0$. Since $\Gamma\subset \phi^*(\mm_{9, 5}^1)$ is a covering curve,  this finishes the proof.
\end{proof}

Finally, we turn to the case of genus $7$.
In order to establish the uniruledness of $\mm_{7, n}$, we consider the following  effective divisors on $\mm_{7, n}$:
$$\cD_1:=\{[C, x_1, \ldots, x_n]\in \cM_{7, n}: \exists L\in W^2_7(C)\mbox{  with } h^0(C, L(-x_1-x_2))\geq 1\}, $$
and $D_2:=\frac{3}{2c_{7, 1, 4}} \phi^*(\mm_{7, 4}^1)\equiv 15\lambda-2\delta_0-9\delta_1-15\delta_2-18\delta_3 \in \mathrm{Pic}(\mm_7)$
is (a rational multiple of) the divisor of $4$-gonal curves on $\mm_7$. Before computing the class $[\dd_1]$,  we need a calculation, which may be of independent interest:

\begin{proposition}\label{divm72}
Let $g\equiv 1  \ \mathrm{ mod } \ 3$ be a fixed integer and set $d:=(2g+7)/3$, so that the Brill-Noether number $\rho(g, 2, d)=1$. One considers the effective divisor of nodes of plane curves
$$\mathfrak{Node}_g:=\{[C, x, y]\in \cM_{g, 2}: \exists L\in W^2_d(C) \mbox{ such that } \ h^0(C, L(-x-y))\geq 2\}.$$
The class of its closure in $\mm_{g, 2}$ is given by the formula:
$$\overline{\mathfrak{Node}}_g\equiv c_g\Bigl((g+4)\lambda+\frac{g+2}{6}(\psi_1+\psi_2)-\frac{g+2}{6}\delta_{\mathrm{irr}}-g\delta_{0: \{1, 2\}}-\cdots\ \Bigr)\in \mathrm{Pic}(\mm_{g, 2}),$$
$$\mbox{ where }\ \  c_g:=\frac{24(g-2)!}{(g-d+5)!\ (g-d+3)!\ (g-d+1)!}.$$
\end{proposition}
\begin{proof} We denote by $\phi_1:\mm_{g, 2}\rightarrow \mm_{g, 1}$ the morphism forgetting the second marked point. The divisor $\overline{\mathfrak{Cu}}_g:=(\phi_1)_*(\overline{\mathfrak{Node}}_g\cdot \delta_{0: \{1, 2\}})$ coincides with the cusp locus in $\mm_{g, 1}$, that is,
the locus of pointed curves $[C, x]\in \cM_{g, 1}$, such that there exists $L\in W^2_d(C)$ with $h^0(C, L(-2x))\geq 2$.

In order to compute its class, we fix a general elliptic curve $[E, x]\in \mm_{1, 1}$ and consider the map $j:\mm_{g, 1}\rightarrow \mm_{g+1}$,  given by $j([C, x]):=[C\cup_x
E]$. Then $$\overline{\mathfrak{Cu}}_g=j^*(\mm_{g+1, d}^2),$$ where $\mm_{g+1, d}^2$ is the Brill-Noether divisor on $\mm_{g+1}$ consisting of curves with a
$\mathfrak g^2_d$. Since the class $[\mm_{g+1, d}^2]\in \mathrm{Pic}(\mm_{g+1})$ is known, cf. \cite{EH2}, and  $j^*(\lambda)=\lambda$, \ $j^*(\delta_{\mathrm{irr}})=\delta_{\mathrm{irr}}$, $j^*(\delta_1)=-\psi+\delta_{g-1: 1}$, we find the following expression
$$\overline{\mathfrak{Cu}}_g\equiv c_g \Bigl((g+4)\lambda+g\psi-\frac{g+2}{6}\delta_{\mathrm{irr}}-\sum_{i=1}^{g-1} (i+1)(g-i)\delta_{i: 1} \Bigr)\in \mathrm{Pic}(\mm_{g, 1}).$$
Using the formulas $(\phi_1)_*(\lambda\cdot \delta_{0: \{1, 2\}})=\lambda$,\ $(\phi_1)_*(\delta_{0: \{1, 2\}}^2)=-\psi$, $(\phi_1)_*(\delta_{\mathrm{irr}}\cdot \delta_{0: \{1, 2\}})=\delta_{\mathrm{irr}}$ and $(\phi_1)_*(\psi_i\cdot \delta_{0: \{1, 2\}})=0$ for $i=1, 2$, one finds that the
$\delta_{0: \{1, 2\}}$-coefficient of $\overline{\mathfrak{Node}}_g$ equals the $\psi_1$-coefficient of $\overline{\mathfrak{Cu}}_g$, while the
$\lambda, \delta_{\mathrm{irr}}$-coefficients coincide.

One still has to determine the $\psi_1$-coefficient in $[\overline{\mathfrak{Node}}_g]$. To this end, we fix a general point $[C, q]\in
\cM_{g, 2}$ and consider the test curve $C_2:=\{[C, q, y]: y\in C\}\subset \mm_{g, 2}$. Then, $C_2\cdot \psi_1=1$, $C_2\cdot \psi_2=2g-1$ and obviously $C_2\cdot \delta_{0: \{1, 2\}}=1$. On the other hand, $C_2\cdot \overline{\mathfrak{Node}}_g$ equals the number of points $y\in C$, such that for some (necessarily complete and base point free) $L\in W^2_d(C)$, the morphism $$\chi(L, y): L_{|y+q}^{\vee}\rightarrow H^0(C, L)^{\vee} $$ fails to be injective. The map $\chi(L, y)$ globalizes to a morphism of vector bundles over $C\times W^2_d(C)$, and the number in question is the Chern number of the top degeneracy locus of $\chi$ and is computed using \cite{HT}. We omit the details.
\end{proof}

\begin{proposition}\label{dinm7n} The class of the closure of $\cD_1$ in $\mm_{7, n}$ is given by the formula
$$\dd_1\equiv 44\lambda+6(\psi_1+\psi_2)-6\delta_{\mathrm{irr}}-28\delta_{0: \{1, 2\}}-6\sum_{j=3}^n (\delta_{0:\{1, j\}}+\delta_{0:\{2, j\}})-\cdots\in \mathrm{Pic}(\mm_{7, n}).$$
\end{proposition}
\begin{proof} We denote by $\phi_{12}:\mm_{7, n}\rightarrow \mm_{7, 2}$ the morphism retaining the first two marked points. Then $\dd_1=\phi_{12}^*(\overline{\mathfrak{Node}}_7)$, and the conclusion follows from  Proposition \ref{divm72} using the pull-back formulas for generators of $\mbox{Pic}(\mm_{7, 2})$, see e.g. \cite{Log} Theorem 2.3.
\end{proof}

\begin{theorem}
$\mm_{7, n}$ is uniruled for $n\leq 13$.
\end{theorem}
\begin{proof}
We start by constructing a covering curve for $\dd_1$. Choose general points $p_1, \ldots, p_8$, $x_3, \ldots, x_{12}\in \PP^2$, and a general line $l\subset \PP^2$. Then consider the pencil of plane septics of geometric genus $7$ passing through $x_3, \ldots, x_{12}$ and having nodes at $p_1, \ldots, p_8$. Blowing-up the nodes
as well as the base points of the pencil, we obtain a fibration $f:S\rightarrow \PP^1$, where $S:=\mathrm{Bl}_{25}(\PP^2)$. We observe that $f$ has sections $\{E_{x_i}\}_{i=3, \ldots, 12}$, given by the respective exceptional divisors, a $2$-section given by $E_{p_1}$ and a $7$-section induced by the proper transform of $l$. We make base changes of order $2$ and $7$ respectively, to arrive at the $1$-cycle
$\Gamma_1:=\bigl\{[C_t, \ x_1(t), \ldots, x_{13}(t)]:t\in \PP^1\bigr\}\subset \mm_{7, 13}$,
where $x_1(t)$ and $x_2(t)$ map to the fixed node $p_1\in \PP^2$, whereas the image of $x_{13}(t)$ lies on the line $l$. One finds that:
$$
\Gamma_1 \cdot \lambda= 14\cdot g=98, \ \Gamma_1 \cdot \psi_1=\Gamma_1 \cdot \psi_2=35, \ \Gamma_1 \cdot \psi_{3}=\cdots =\Gamma_1 \cdot \psi_{{12}}=14,\ \Gamma_1\cdot \psi_{13}=22.
$$
Furthermore, $\Gamma_1 \cdot \delta_{0: \{1, 2\}}=14,\  \Gamma_1\cdot \delta_{\mathrm{irr}}=14\cdot 52=728$, and finally $\Gamma_1\cdot \delta_{j: T}=0$ for all pairs $(j, T)\neq \bigl(0, \{1, 2\}\bigr)$. Clearly  $\Gamma_1$ is a covering curve for $\dd_1$.
\vskip 3pt

Next, we construct a covering curve for $D_2$ and use that if $[C]\in \cM_{7, 4}^1$ and $A\in W^1_4(C)$ is the corresponding pencil, then there exists a divisor $D\in C_3$ such that $A\otimes \OO_C(D)\in W^2_7(C)$. One fixes general points $p, \{p_i\}_{i=1}^5$, $\{x_j\}_{j=1}^{13} \in \PP^2$ and considers the pencil of genus $7$ septics with a $3$-fold point at $p$, nodes at
$p_1, \ldots, p_5$ and passing through $x_1, \ldots, x_{13}$. This induces a covering curve $\Gamma_2\subset \phi^*(\mm_{7, 4}^1)$ with the following  invariants:
$$\Gamma_2\cdot \lambda=7, \ \Gamma_2\cdot \delta_{\mathrm{irr}}=53, \ \Gamma_2\cdot \psi_i=1 \mbox{ for } i=1, \ldots, 13, \ \mbox{ and } \  \Gamma_2\cdot  \delta_{j: T}=0\ \ \mbox{ for all }\  (j, T).$$
Thus,  $\Gamma_1\cdot \dd_1=-28, \  \Gamma_2\cdot D_2=-1, \ \Gamma_1\cdot D_2=14,  \Gamma_2\cdot \dd_1=2$,  as well as
$\Gamma_1\cdot K_{\mm_{7, 13}}=22$ and $\Gamma_2\cdot K_{\mm_{7, 13}}=-2$. The assumptions of Proposition \ref{uniruled} are thus fulfilled.
\end{proof}

\end{document}